\documentclass[11pt]{amsart}
\usepackage{amssymb,hyperref}

\newtheorem{theorem}{Theorem}[section] 

\newtheorem{conjecture}[theorem]{Conjecture}
\newtheorem{corollary}[theorem]{Corollary}
\newtheorem{definition}[theorem]{Definition}
\newtheorem{example}[theorem]{Example}

\newtheorem{fact}[theorem]{Fact}
\newtheorem{lemma}[theorem]{Lemma}
\newtheorem{problem}[theorem]{Problem}
\newtheorem{proposition}[theorem]{Proposition}
\newtheorem{question}[theorem]{Question}
\newtheorem{remark}[theorem]{Remark}

\newcommand{\bcon}{\begin{conjecture}}
\newcommand{\econ}{\end{conjecture}}
\newcommand{\bcor}{\begin{corollary}}
\newcommand{\ecor}{\end{corollary}}
\newcommand{\bdf}{\begin{definition}}
\newcommand{\edf}{\end{definition}}
\newcommand{\beq}{\begin{equation}}
\newcommand{\eeq}{\end{equation}}
\newcommand{\bexa}{\begin{example}}
\newcommand{\eexa}{\end{example}}
\newcommand{\bfac}{\begin{fact}}
\newcommand{\efac}{\end{fact}}
\newcommand{\blem}{\begin{lemma}}
\newcommand{\elem}{\end{lemma}}
\newcommand{\bprb}{\begin{problem}}
\newcommand{\eprb}{\end{problem}}
\newcommand{\bpro}{\begin{proposition}}
\newcommand{\epro}{\end{proposition}}

\newcommand{\bque}{\begin{question}}
\newcommand{\eque}{\end{question}}
\newcommand{\brem}{\begin{remark}}
\newcommand{\erem}{\end{remark}}
\newcommand{\bthm}{\begin{theorem}}
\newcommand{\ethm}{\end{theorem}}

\newcommand{\be}{\begin}
\newcommand{\en}{\end}
\newcommand{\bpr}{\begin{proof}}
\newcommand{\epr}{\end{proof}}

\newcommand{\M}[1]{\operatorname{M(#1,\C)}}
\newcommand{\SL}[1]{\operatorname{SL(#1,\C)}}
\newcommand{\PSL}[1]{\operatorname{PSL(#1,\C)}}
\newcommand{\GL}[1]{\operatorname{GL(#1,\C)}}
\newcommand{\SO}[1]{\operatorname{SO(#1,\C)}}
\newcommand{\Sp}[1]{\operatorname{Sp(#1,\C)}}

\renewcommand{\sl}[1]{\mathfrak{sl(#1,\C)}}

\newcommand{\so}[1]{\mathfrak{so(#1,\C)}}
\renewcommand{\sp}[1]{\mathfrak{sp(#1,\C)}}

\newcommand{\lb}{\label}

\newcommand{\comment}[1]{\,}

\newcommand{\cal}{\mathcal}

\newcommand{\Z}{\mathbb Z}

\newcommand{\C}{\mathbb C}

\newcommand{\X}{\mathcal X}

\newcommand{\g}{\mathfrak g}

\renewcommand{\t}{\mathfrak t}

\newcommand{\T}{\mathbb T}

\newcommand{\ve}{\varepsilon}
\usepackage{epsfig}

\title{Character Varieties of Abelian Groups}

\author{Adam S. Sikora}
\thanks{The author acknowledges support from U.S. National Science Foundation grants DMS 1107452, 1107263, 1107367 "RNMS: GEometric structures And Representation varieties" (the GEAR Network).}


\begin{document}

\thispagestyle{empty}

\begin{abstract} We prove that for every reductive group $G$ with a maximal torus $\T$ and the Weyl group $W$, $\T^N/W$ is the normalization of the irreducible component, $X_G^0(\Z^N),$ of the $G$-character variety $X_G(\Z^N)$ of $\Z^N$ containing the trivial representation. We also prove that $X_G^0(\Z^N)=\T^N/W$ for all classical groups.

Additionally, we prove that even though there are no irreducible representations in $X_G^0(\Z^N)$ for non-abelian $G$, the tangent spaces to $X_G^0(\Z^N)$ coincide with $H^1(\Z^N, Ad\, \rho).$ Consequently, $X_G^0(\Z^2),$ has the ``Goldman" symplectic form for which the combinatorial formulas for Goldman bracket hold.
\end{abstract}

\pagestyle{myheadings}

\maketitle

%
\section{Introduction}
%

Let $G$ will be an affine reductive algebraic group over $\C.$\footnote{The field of complex numbers can be replaced an arbitrary algebraically closed field of zero characteristic throughout the paper.} For every finitely generated group $\Gamma$, the space of all $G$-representations of $\Gamma$ forms an algebraic set, $Hom(\Gamma,G),$ on which $G$ acts by conjugating representations.
The categorical quotient of that action $$X_G(\Gamma)=Hom(\Gamma,G)//G$$ is the $G$-character variety of $\Gamma,$ cf. \cite{LM, S-char} and the references within.
In this paper we study $G$-character varieties of free abelian groups.

For a Cartan subgroup (a maximal complex torus) $\T$ of $G$, the map
$$\T^N=Hom(\Z^N,\T)\to Hom(\Z^N,G)\to Hom(\Z^N,G)//G=X_G(\Z^N)$$ factors through
\beq\lb{e-chi}
\chi: \T^N/W\to X_G(\Z^N),
\eeq
where the Weyl group $W$ acts diagonally on $\T^N=\T\times ... \times \T$.
Thaddeus proved that for every reductive group $G$, $\chi$ is an embedding, \cite{Th}.
In this paper we discuss the image of this map and the conditions under which it is an isomorphism.
This is known to be a difficult problem. A version of it for compact groups is discussed for example in \cite{BFM}. (The connections between the algebraic and compact versions of this problem are discussed in \cite{FL}.) The version of this problem for algebraic groups is harder than that for compact ones, since a regular bijective function between algebraic varieties does not have to be an algebraic isomorphism.

Goldman constructed a symplectic form on an open dense subset of the set of equivalence classes of irreducible representations in $X_G(\pi_1(F)),$ for closed surfaces $F$ of genus $>1$, \cite{Go1}. In the second part of the paper, we extend Goldman's construction to the connected component of the identity of the $G$-character variety of $F$ torus, even though there are no irreducible representations in that component.

This paper was motivated by \cite{Th} and by our work, \cite{S-qt}, in which we relate deformation-quantizations of character varieties of the torus to the $q$-holonomic properties of Witten-Reshetikhin-Turaev knot invariants.

%
\section{Main results}
%

Let $X^0_G(\Z^N)=\chi(\T^N/W).$

\bthm[Proof in Sec. \ref{s_proofs}]\ \\ \lb{main}
(1) $X^0_G(\Z^N)$ is an irreducible component of $X_G(\Z^N)$.\\
(2) $\chi:\T^N/W\to X_G^0(\Z^N)$ is a normalization map for every $G$ and $N.$ (It was proved for $N=2$ in \cite{Th}.)\\
(3) $\chi: \T^N/W\to X_G^0(\Z^N)$ is an isomorphism for classical groups: $G=\GL{n},$ $\SL{n}, \Sp{n}, \SO{n}$ and for every $n$ and $N.$ ($\Sp{n}$ denotes the group of $2n\times 2n$ matrices preserving a symplectic form.)
\ethm

\brem 
(1) It is easy to show that $\chi$ is onto for $G=\SL{n}$ and $\GL{n},$ since one can conjugate every $G$-representation of $\Z^N$ arbitrarily close to representations into $\T$. That does not hold though for some other groups $G.$ For example, a representation sending $\Z^n$ onto the group of diagonal matrices $D$ in $\operatorname{O}(n,\C)=\{A: A\cdot A^T=I\}$ ($D=\{\pm 1\}^n$) for $n>3$ cannot be conjugated arbitrarily close to a representation into a maximal torus.

(2) $\chi$ being onto and $1$-$1$ does not imply that it is an isomorphism of algebraic sets. (For example, $x\to (x^2,x^3)$ from $\C$ to $\{(x,y): x^3=y^2\}\subset \C^2$ is a bijection which is not an isomorphism.)
\erem

\bprb 
Is $\chi$ is an isomorphism onto its image for Spin groups and the exceptional ones?
(By Theorem \ref{main}(2), $\chi$ is an isomorphism if and only if $X_G^0(\Z^N)$ is normal.)
\eprb


Here are a few basic facts about irreducible and connected components of $X_G(\Z^N)$.

\brem \lb{ir=con}
(1) $X_G(\Z)$ is irreducible, cf. \cite[\S 6.4]{St}.\\
(2) $X_G(\Z^2)$ is irreducible for every semi-simple simply-connected group $G,$ cf. \cite[Thm C]{Ric}.\\
(3) For every connected $G$, $X_G^0(\Z^2)$ coincides with the connected component of the trivial representation in $X_G(\Z^2),$ cf.\cite{Th}. (For completeness, a proof is enclosed in Sec. \ref{s_proofs2}.)
\erem

\bpro[Proof in Sec. \ref{s_proofs2}]\lb{irred}\ \\
$X_G(\Z^N)$ is irreducible for $G=\GL{n}, \SL{n}$ and $\Sp{n}$ for all $N$ and $n.$
\epro


%
\section{Irreducible representations of $\Z^N$}
\lb{s_irep}
%

Following \cite{S-char}, we say that $\rho:\Z^N\to G$ is irreducible if its image does not lie in a proper parabolic subgroup of $G.$ We say that $\rho:\Z^N\to G$ is completely reducible if for every parabolic subgroup $P\subset G$ containing $\rho(\Z^N)$, the image of $\rho$ lies in a Levi subgroup of $P$.

\bpro \lb{irrep}
For non-abelian $G$ there are no irreducible representations $\rho: \Z^N\to G$ with $[\rho]\in X^0_G(\Z^N).$
\epro

\bpr Assume that $\rho$ is irreducible and $[\rho]\in X^0_G(\Z^N).$ Every equivalence class in $X_G^0(\Z^N)$ contains a
representation $\phi: \Z^N\to \T\subset G$, and such $\phi$ is completely reducible. Since $\rho$ (being irreducible) is completely reducible and each equivalence class in $X_G(\Z^N)$ contains a unique conjugacy class of completely reducible representation, $\rho$ is conjugate to $\phi$. Hence, $\phi$ is irreducible. Therefore, $G=\T$, contradicting the assumption of $G$ being non-abelian.
\epr

\bcor
There are no irreducible representations of $\Z^2$ into simply-connected reductive groups.
\ecor

\bpr Every simply connected reductive algebraic Lie group is semi-simple. (That follows for example from two facts:  1. every reductive Lie algebra is a product of a semi-simple one and an abelian one. 2. There are no non-trivial simply-connected abelian reductive algebraic groups.) Now the statement follows from Remark \ref{ir=con}(2) and (3).
\epr

There are, however, irreducible representations of abelian groups into non-abelian ones.

\bexa 
Let $n\geq 3$ and let $\rho: \Z^{N}\to \SO{n}$ be a representation whose image contains all diagonal orthogonal matrices with entries $\pm 1$ in the diagonal. Then $\rho$ is irreducible, cf. \cite[Eg. 21]{S-char}.
\eexa

Another example was suggested to us by Angelo Vistoli:

\bexa
Let $g\in \PSL{n}$ be represented by the diagonal matrix with $1,\omega^1,...,\omega^n$ on the diagonal, where $\omega=e^{2\pi i/n},$ and let $h$ be represented by the permutation matrix associated with the cycle $(1,2,...,n).$
Then it is easy to see that $g$ and $h$ commute and to prove that $\rho: \Z^2\to \PSL{n}$ sending the generators of $\Z^2$ to $g$ and $h$ is irreducible.
\eexa

%
\section{\'Etale Slices and Chevalley sections}
%

Let $Hom^0(\Z^N,G)$ be the preimage of $X^0_G(\Z^N)$ under $\pi: Hom(\Z^N,G)\to X_G(\Z^N).$

\bthm[Proof in Sec. \ref{s_proofs3}]\lb{slice} If $\rho:\Z^N\to \T\subset G$ has a Zariski dense image in $\T$ then\\
(1) $X_G(\Z^N)$ is smooth at $\rho$ and $\rho$ belongs to a unique irreducible component of $X_G(\Z^N)$.\\
(2) $d\chi: T_\rho \T^N/W\to T_\rho\, X_G(\Z^N)$ is an isomorphism.\\
(3) A Zariski open neighborhood of $\rho$ in $\T^N=Hom(\Z^N,\T)$ is an \'etale slice at $\rho$ with respect to the $G$ action on $Hom^0(\Z^N,G)$ by conjugation.
\ethm

With $Hom(\Gamma,G)$ and $X_G(\Gamma),$ there are naturally associated algebraic schemes ${\cal Hom}(\Gamma,G)$ and ${\cal X}_G(\Gamma)={\cal Hom}(\Gamma,G)//G$ such that the coordinate rings, $\C[Hom(\Gamma,G)]$ and $\C[X_G(\Gamma)],$ are nil-radical quotients of the algebras of global sections of ${\cal Hom}(\Gamma,G)$ and of ${\cal X}_G(\Gamma),$ cf. \cite{S-char}.

For every completely reducible $\rho:\Z^N\to G$ there exists a natural linear map
\beq\lb{e_H^1-T}
\phi: H^1(\Z^N, Ad\, \rho)\to T_{[\rho]}\, \X_G(\Z^N)
\eeq
defined explicitly in \cite[Thm. 53]{S-char}, where the cohomology group has coefficients in the Lie algebra $\g$ of $G$, twisted by $\rho$ composed with the adjoint representation of $G.$
Although this map is not an isomorphism in general,
it is known to be one for good $\rho,$ cf. \cite[Thm. 53]{S-char}. As we have seen in the previous section, there are no irreducible representations in $X_G^0(\Z^N)$. Nonetheless, Theorem \ref{slice} implies the following result which will be used in Section \ref{s_symp}:

\bcor[Proof in Sec. \ref{s_proofs3}]\ \\ \lb{T=H}
For every $\rho: \Z^N\to \T\subset G$  such that $\rho(\Z^N)$ is Zariski dense in $\T$, the map (\ref{e_H^1-T}) is an isomorphism.
\ecor

One says that a subvariety $S$ of an algebraic variety $X$ is a Chevalley section with respect to a $G$-action on $X$,
if the natural map $S//N(S)\to X//G$ is an isomorphism, where $N(S)=\{g\in G: gS=S\}$, cf. \cite[Sec 3.8]{PV}.
For example, any maximal torus in $G$ is a Chevalley section of $G$ with respect to the $G$-action by conjugation.

The crucial question of whether $\chi: \T^N/W\to X_G^0(\Z^N)$ is an isomorphism is equivalent to the question whether $Hom(\Z^N,\T)$ is a Chevalley section of $Hom^0(\Z^N,G)=\pi^{-1}(X_G^0(\Z^N))$ under the $G$ action by conjugation.

%
\section{More on connected components of $X_G(\Z^N)$ for semi-simple $G$}
%

Assume now that $G$ is semi-simple. Then $\pi_1(G)$ is finite and the central extension
$$\{e\} \to \pi_1(G)\to \bar G\to G\to \{e\},$$
where $\bar G$ is the universal cover of $G$, defines an element $\tau\in H^2(G,\pi_1(G))$, cf. \cite[Thm IV.3.12]{Br}. (Since the extension is central, the action of $G$ on $\pi_1(G)$ is trivial.) Hence, every representation $\rho: \Z^N\to G$ defines $\rho^*(\tau)\in H^2(\Z^N,\pi_1(G)).$
By the universal coefficient theorem,
$$H^2(\Z^N,\pi_1(G))=Hom(H_2(\Z^N),\pi_1(G))=Hom(\Z^{N\choose 2},\pi_1(G))=\pi_1(G)^{N \choose 2}.$$

The map $\rho \to \rho^*(\tau)$ is continuous on $Hom(\Z^N,G)$ and it is invariant under the conjugation by $G$. Therefore, its restriction to completely reducible representations $Hom^{cr}(\Z^N,G)\subset Hom(\Z^N,G)$ factors through a continuous map $Hom^{cr}(\Z^N,G)/G=X_G(\Z^N)\to \pi_1(G)^{N \choose 2}.$
Since this map is constant on connected components of $X_G(\Z^N)$, it yields
$$\Psi: \pi_0(X_G(\Z^N))\to H^2(\Z^N,\pi_1(G)).$$
\bpro \lb{Z^2-comp}
$\Psi$ is a bijection for $G$ connected and $N=2.$
\epro

Following \cite{BFM}, we say that $(g_1,g_2)\in {\bar G}^2$ is a $c$-pair if $[g_1,g_2]=c \in C(\bar G),$ the center of $\bar G.$

\noindent {\it Proof of Proposition \ref{Z^2-comp}:}
Since $G$ is connected, the connected components of $Hom(\Z^2,G)$ are in a natural bijection with those of $X_G(\Z^N).$ Let $K$ be the compact form of $G$. By \cite{FL}, the map $Hom(\Z^2,K)\to Hom(\Z^2,G)$ induced by the embedding $K\to G$ is a bijection on connected components. By Cartan decomposition, $K$ is a deformation retract of $G$ and, consequently, $\pi_1(K)=\pi_1(G).$ Therefore, it is enough to show that the corresponding map $\pi_0(Hom(\Z^2,K))\to \pi_1(K)$ is a bijection.   The representations $\rho:\Z^2\to K$ with $\Psi(\rho)=c\in H^2(\Z^2,\pi_1(K))=\pi_1(K)\subset C(\bar K)$ correspond to $c$-pairs in $\bar K,$ cf. \cite{BFM}. By \cite[Thm. 1.3.1]{BFM}, the space of $c$-pairs for any given $c\in \pi_1(K)$ is non-empty and connected.
\qed\ \\

$\Psi$ is a bijection between $\pi_0(X_G(\pi_1(F)))$ and $\pi_1(G)$ for closed orientable surfaces $F$ of genus $>1$ as well, cf. \cite{Li}.

$\Psi$ is generally not $1$-$1$ for $N>2.$ For example, $X_G(\Z^N)$ is disconnected for
$N>2$ and for all simply-connected groups $G$ other than the products of $\SL{2}$ and of $\Sp{n},$  \cite{FL,KS}.

Denote by $X_G^c(\Z^2)$ the connected component of $X_G(\Z^2)$ with the $\Psi$-value $c\in \pi_1(G).$
Identify $\pi_1(G)$ with a subgroup of the center of $\bar G,$ $C(\bar G).$  The group $G$ acts on on the set of all $c$-pairs, $M_G^c\subset {\bar G}^2,$ and it is easy to see that the natural map
$$\phi_c: M_G^c//{\bar G}\to X_G(\Z^2)$$
is a finite algebraic map.
Let us analyze $M_G^c//{\bar G}$ further following the approach of \cite{BFM}:
Any element $c\in C(\bar G)$ acts on $\T.$ Let $S^c\subset \T$ be the connected component of identity in the invariant part of $c$ action on $\T.$ Let $S'$ be the subtorus of $\T$ determined by the orthogonal component of the Lie algebra of $S$ in the Lie algebra of $\T,$ with respect to the Killing form.
Then $F_S=S\cap S'$ is a finite group.
Following \cite[Thm 1.3.1]{BFM}, it is easy to show that there is a regular map,
$$\theta_c: ((S/F_S)\times (S/F_S))/W\to M_G^c,$$
where $W$ is the quotient of the normalizer of $S$ by its centralizer in $G.$

\bprb 
Is $M_G^c$ is irreducible? Is $\theta_c$ a normalization map? Is it an isomorphism?
\eprb

%
\section{Symplectic nature of the character varieties of the torus}
\lb{s_symp}
%

Goldman constructed a symplectic form on the set of equivalence classes of ``good" representations in $X_G(\pi_1(F))$ for every reductive $G$ and for every closed orientable surface $F.$
Motivated by applications to quantum topology, \cite{S-qt}, we are going to extend his construction to tori.

Goldman's approach relies on identifying the tangent space, $T_{[\rho]}\, X_G(\pi_1(F))$, at an irreducible $\rho$ with $H^1(F, \g),$ where $\g$ is the Lie algebra of $G$, and the coefficients in this cohomology are twisted by $Ad\, \rho$, cf. \cite{Go1, Go2}. (One needs an additional assumption that the stabilizer of $\rho(\pi_1(F))\subset G$ coincides with the center of $G$, \cite{S-char}.)
Although his construction does not extend to $X_G^0(\Z^2)$ (i.e. torus), since, as shown in Sec. \ref{s_irep}, no representation in that component of character variety is irreducible for non-abelian $G$, we resolved that difficulty with our Corollary \ref{T=H}.

Let $X_G'(\Z^N)=Hom'(\Z^N,G)//G$, where $Hom'(\Z^N,G)$ is the space of $G$-representations of $\Z^N$ with a Zariski dense image in a maximal torus of $G.$ (Since all representations in $Hom'(\Z^N,G)//G$ are completely reducible, it is the set-theoretic quotient.)
Let $N=2$.  The composition of the cup product
$$H^1(\Z^2, Ad\, \rho)\times H^1(\Z^2, Ad\, \rho)\to H^1(\Z^2, Ad\, \rho\otimes Ad\, \rho)$$
with the map
$$H^1(\Z^2, Ad\, \rho\otimes Ad\, \rho)\to H^2(\Z^2,\C)=\C$$
induced by an $Ad\, G$-invariant  symmetric, non-degenerate bilinear form,
$$\frak B:\g \times \g\to \C,$$
defines a skew-symmetric pairing
\beq\lb{e-omega}
\omega: H^1(\Z^2, Ad\, \rho)\times H^1(\Z^2, Ad\, \rho)\to \C.
\eeq
By Corollary \ref{T=H}, this pairing defines a differential $2$-form on $X_G'(\Z^N)$.
We claim that $\omega$ is symplectic. Let us precede the proof with a construction of another closely related form:
Let $\omega'$ be the $2$-form on $T_{(e,e)} \T\times \T=\t\times \t$ defined by
$$\omega'((v_1,w_1),(v_2,w_2))= \frak B(v_1,w_2)-\frak B(v_2,w_1).$$  It defines an invariant $2$-form on $\T\times \T$ which descends to a non-degenerate skew-symmetric $2$-form on $(\T\times \T)/W$. (Recall that $W$ acts on $\T$ and, by extension, it acts diagonally on $\T\times \T.$)

\bpro[Proof in Sec. \ref{s_proofs3}]\ \\ \lb{om=om'}
\noindent (1) The pullback of $\omega$ through $\chi: \T^2/W\to X_G^0(\Z^2)$ coincides with $\omega'$.\\
(2) Both $\omega$ and $\omega'$ are symplectic.
\epro

The most obvious choice for $\frak B$ is the Killing form, $\frak K.$ However, it is also useful to consider the trace form $\frak T(A,B)=Tr(AB)$ for classical Lie algebras $\g$ with their natural representations by matrices, $\sl{n},\so{n}\subset \M{n},$ and $\sp{n}\subset \M{2n}$. In that case, $\frak K=c_{\g} \cdot \frak T,$ where
$$c_{\sl{n}}=2n,\quad c_{\so{n}}=n-2,\quad c_{\sp{n}}=2n+2.$$

Our construction of $\omega$ is an exact analogue of that of Goldman's symplectic form for character varieties of surfaces of higher genera, except for the fact that it is a holomorphic form (defined using the form $\frak B$) rather than a real form (defined by the real part of $\frak B$), cf. \cite{S-char}. Therefore, it is not surprising to see most of the methods and results of \cite{Go2} apply to character varieties of tori as well, cf. our Proposition \ref{Goldman-formulas}. For example, here is a version of Goldman's combinatorial formulas for Poisson brackets for the character varieties of the torus.

\begin{proposition}[Proof in Sec. \ref{s_proof_symp}]\label{G-brack}
Let $\{\cdot,\cdot\}$ be the Poisson bracket on $\C[X_G^0(\Z^2)]$ induced by $\omega$ defined by a form $\frak B=c \cdot \frak T$, where $\frak T$ is the trace form and $c\in \C^*$. Let $\tau_g: X_G^0(\Z^2)\to \C$ be defined by $\tau_g([\rho])=Tr \rho(g).$ Then for any $p,q,r,s\in \Z,$
$$\{\tau_{(p,q)},\tau_{(r,s)}\}= \frac{1}{c} \left|\begin{array}{cc} p & q \\ r & s\\ \end{array}\right|
\left(\tau_{(p+r,q+s)}-\frac{\tau_{(p,q)}\tau_{(r,s)}}{n}\right),\ \text{for\ } G=\SL{n},$$
and
$$\{\tau_{(p,q)},\tau_{(r,s)}\}= \frac{1}{2c}\left|\begin{array}{cc} p & q \\ r & s\\ \end{array}\right|
\left(\tau_{(p+r,q+s)}-\tau_{(p-r,q-s)}\right)\ \text{for}\ G=\SO{n}, \Sp{n}.$$
\end{proposition}

%
\section{Proof of Theorem \ref{main}} 
\lb{s_proofs}
%


\bpro[cf. \cite{Th}] \lb{1-1}
(1) $\chi$ is $1$-$1$\\
(2) $\chi$ is finite.
\epro

\bpr (1) The proof is an extension of the arguments of \cite{B1} and of \cite{Th}:
Let $\rho,\rho':\Z^N\to \T$ be equivalent in $X_G(\Z^N).$ We prove first that $\rho$ and $\rho'$ are conjugate. Since the algebraic closures of $\rho(\Z^N)$ and of $\rho'(\Z^N)$
are finite extensions of tori, they are linearly reductive and, hence, by \cite[Prop. 8]{S-char},  $\rho$ and $\rho'$ are completely reducible representations of $\Z^N$ into $G.$ Since the orbit of the $G$-action by conjugation on a completely reducible representation in $Hom(\Z^N,G)$ is closed, cf. \cite[Thm. 30]{S-char}, we see that $\rho'=g\rho g^{-1},$ for some $g\in G.$
The centralizer of $\rho(\Z^N),$ $Z(\rho(\Z^N))\subset G$ is a reductive group
by \cite[26.2A]{Hu} since the proof there is valid not only for a subtorus but for any
subset. Clearly $\T\subset Z(\rho(\Z^N)).$ Since the elements of $\T$ commute with elements of $\rho'(\Z^N)=g\rho(\Z^N) g^{-1}$,
the elements of $g^{-1}\T g$ commute with those of $\rho(\Z^N)$ and, hence, $g^{-1}\T g \subset Z(\rho(\Z^N)).$
Since $\T$ and $g^{-1}\T g$ are maximal tori in $Z(\rho(\Z^N)),$ there is $h\in Z(\rho(\Z^N))$ such that
$h^{-1}g^{-1}\T gh=\T.$ This conjugation on $\T$ coincides with the action of an element $w$ of the Weyl group on $\T.$
Since
$$w\cdot \rho'(x)=h^{-1}g^{-1}\rho'(x) gh=h^{-1}\rho(x) h=\rho(x)$$
for every $x\in \Z^N,$ the statement follows.

(2) We follow \cite{Th}: The map $\T^N/W \to \T/W\times ...\times \T/W$ is finite. Since it factors through
$$\T^N/W\to X_G(\Z^N)\to X_G(\Z)\times ...\times X_G(\Z)\to \T/W\times ...\times \T/W,$$
and the right map is an isomorphism, the map $\T^N/W\to X_G(\Z^N)$ is finite, cf. \cite[Lemma 2.5]{Ka}.
\epr

\begin{lemma}\label{H^1-Z^N}
(1) For every non-trivial homomorphism $\psi:\Z^N\to \C^*,$\\
$H^1(\Z^N,\psi)=0.$\\
(2) Let $\g$ and $\t$ be the Lie algebras of $G$ and of a maximal torus $\T$ in $G$, respectively. If $\rho: \Z^N\to \T$ is a representation whose image does not lie in $Ker\, \alpha$, for any root $\alpha$ of $\g$, then the embedding $\t\subset \g$ induces an isomorphism
$$\t^{N}=H^1(\Z^N,\t)\to H^1(\Z^N, Ad\, \rho).$$
\end{lemma}

\begin{proof}
(1) The first cohomology group is the quotient of the space of derivations
\begin{equation}\label{Z-der}
\sigma: \Z^N\to \C,\quad \sigma(a+b)=\sigma(a)+\psi(a)\sigma(b)
\end{equation}
by the principal derivations,
$$\sigma_m(a)=(\psi(a)-1)\cdot m,$$
for some $m\in \C.$

If $\psi(v)\ne 0$, for some $v\in \Z^N$ then for every $w\in \Z^N,$
$$\sigma(v)+\psi(v)\sigma(w)=\sigma(v+w)=\sigma(w)+\psi(w)\sigma(v).$$
Hence $$\sigma(w)=(\psi(w)-1)\sigma(v)/(\psi(v)-1)$$ and
$\sigma$ is the principal derivation $\sigma_m$ for $m=\sigma(v)/(\psi(v)-1).$

(2) Consider a root decomposition of $\g$,
$$\g=\t\oplus \bigoplus_\alpha \g_\alpha,$$
where the sum is over all roots of $\g$ relative to $\t$ and $\g_\alpha$'s are root subspaces of $\g,$
\cite[8.17]{B2}. Since the image of $\rho$ lies in $\T$, this root decomposition is $Ad\, \rho$ invariant.
Therefore, $$H^1(\Z^N, Ad\, \rho)=H^1(\Z^N,(\t)_{Ad \rho})\oplus \bigoplus_\alpha H^1(\Z^N,(\g_\alpha)_{Ad \rho}).$$
The $Ad\, \rho$ action on $\t$ is trivial. On the other hand, every $v \in \Z^N$ acts on $\g_\alpha$ by the multiplication by  $\alpha(\rho(v)).$ Now the statement follows from (1).
\end{proof}

\noindent {\bf Proof of Theorem \ref{main}:}
(1) By Prop \ref{1-1}(2), $\chi$ is finite and, hence, its image is closed. Since $\T^N/W$ is irreducible, also $X_G^0(\Z^N)=\chi(\T^N/W)$ is irreducible and, consequently, it is contained in an irreducible component $Z$ of $X_G(\Z^N)$.
It is enough to show that $dim\, Z=dim\, X_G^0(\Z^N).$

Consider a representation $\rho:\Z^N\to \T\subset G$ with a Zariski dense image in $\T$.
We have
$$T_\rho\, {\cal Hom}(\Z^N,G)\simeq Z^1(\Z^N, Ad\, \rho)\simeq H^1(\Z^N,Ad\, \rho)\oplus B^1(\Z^N,Ad\, \rho),$$ by \cite[Thm 35]{S-char}.
The first summand has dimension $N\cdot rank\, G$, by Lemma \ref{H^1-Z^N}. The second, composed of functions $\sigma_v: \Z^N\to \g$ of the form
$$\sigma_v(w)=(Ad\, \rho(w)-1)v,$$
for some $v\in \g$, has dimension $\dim\, \g-rank\, \g.$ (Indeed, $\sigma_v=0$ for $v\in \t$ while $\sigma_v$'s are linearly independent for basis elements $v$ of a subspace of $\g$ complementary to $\t.$)
As before, let $Hom^0(\Z^N,G)=\pi^{-1}(X_G^0(\Z^N)).$ Then
\beq\lb{e-T}
dim\, Hom^0(\Z^N,G)\leq dim\, T_\rho\, {\cal Hom}(\Z^N,G)=N\cdot rank\, G+dim\, G-rank\, G.
\eeq
$X_G(\Z^N)$ is the quotient of $Hom^0(\Z^N,G)$ by the action of $G$ with the stabilizer of dimension $rank\, G$ at $\rho.$ Since the stabilizer dimension is a upper semi-continuous function, cf. \cite[Sec. 7]{PV},
the stabilizers near $\rho$ have dimensions at least $rank\, G$. Therefore,
\beq\lb{e_1st_in}
dim\, Z\leq dim\, Hom^0(\Z^N,G)-(dim\, G-rank\, G)\leq N\cdot rank\, G.
\eeq
However, since $\chi$ is an embedding of a variety of dimension $N\cdot rank\, G$ into $X_G^0(\Z^N)$,
$$N\cdot rank\, G\leq dim\, X_G^0(\Z^N)\leq dim\, Z.$$
This inequality together with (\ref{e_1st_in}) implies the statement.

(2) By Proposition \ref{1-1}(1) $\chi$ is $1$-$1$. Hence, $\chi: \T^N/W\to X_G^0(\Z^N)$ is birational, by \cite[Prop 3.17]{Mu}.
Since $\chi$ is finite, $\chi$ is a normalization map, cf. \cite[II.\S 5]{Sh}.\\
(3) By Proposition \ref{class_onto}, $\chi_*$ is onto. Since $\T^N/W\to X_G(\Z^N)$ factors though $X_G^0(\Z^N),$ also
$\chi_*: \C[X_G^0(\Z^N)]\to \C[\T^N/W]$ is onto. Since $\chi$ is onto $X_G^0(\Z^N)$, $\chi_*$ is $1$-$1$ and, hence, an isomorphism.
\qed\ \\


The proof of Theorem \ref{main}(3) above relies on the following:

\bpro\lb{class_onto} For $G=\GL{n}, \SL{n}, \Sp{n}, \SO{n}$ and every $n$ and $N$,
the dual map $\chi_*: \C[X_G(\Z^N)]\to \C[\T^N/W]=\C[\T^N]^W$ is onto.
\epro

\bpr
The coordinate ring of a maximal torus in a reductive algebraic group can be identified with the group ring, $\C\Lambda$, of the weight lattice, $\Lambda,$ of $G.$ Following \cite[\S 23.2]{FH}, we have

(a) If $G=\GL{n}$ then
$$\C[\T^N]=\C[x_{ij}^{\pm 1}, 1\leq i\leq n, 1\leq j\leq N].$$

(b) If $G=\SL{n}$ then
$$\C[\T^N]=\C[x_{ij}^{\pm 1}, 1\leq i\leq n, 1\leq j\leq N]/ I,$$
where $I$ is generated by $\prod_{i=1}^n x_{ij}-1$ for $1\leq j\leq N.$

(c) If $G=\Sp{n}$ then
$$\C[\T^N]=\C[x_{ij}^{\pm 1}, 1\leq i\leq n, 1\leq j\leq N]/I.$$

(d) If $G=\SO{2n}$ then
$$\C[\T^N]=\C[x_{ij}^{\pm 1}, (x_{1j}\cdot ...\cdot x_{nj})^\frac{1}{2}, 1\leq i\leq n, 1\leq j\leq N].$$

(e) If $G=\SO{2n+1}$ then
$$\C[\T^N]=\C[x_{ij}^{\pm 1}, 1\leq i\leq n, 1\leq j\leq N].$$
(Note that $(x_{1j}\cdot ...\cdot x_{nj})^\frac{1}{2}$ is a weight of $Spin(2n+1,\C)$ but not of $\SO{2n+1}$.)


Hence each monomial in variables $x_{ij}$ is of a form
$$m=\prod_{i=1}^n x_i^{\alpha_i},$$
where $\alpha_i=(\alpha_{i1},...,\alpha_{iN})$ are in $\Z^N$ for $G=\GL{n}, \SL{n}, \Sp{n},\SO{2n+1}$
and in $(\frac{1}{2}\Z)^N$ for $G=\SO{2n}.$
(For $G=\SL{n}, \Sp{n}$ such presentation of $m$ is not unique.)

We say that $m$ is a monomial of level $l$ if it has a presentation with $l$ non-vanishing alphas, $\alpha_{i_1},...,\alpha_{i_l}$,
and it has no presentation with $l-1$ non-vanishing alphas. We say that an element of $\C[\T^N]$ is of level $l$ if it is a linear combination of monomials of level $\leq l$ but not a linear combination of monomials of level $<l.$

In each of the above cases, the Weyl group is a subgroup of the signed symmetric group, $SS_n=S_{n}\rtimes (\Z/2)^n,$ and the Weyl group action on $\C[\T^N]$ extends to that of $SS_n$ on $\C[\T^N]$ by permuting the first indices of $x_{ij}$ and negating exponents of these variables, depending on the value of $i.$

Let $\tau_\alpha$, for $\alpha\in \Z^N$, be the function on $\C[X_G(\Z^N)]$ sending the equivalence class of
$\rho:\Z^N\to G$ to $Tr(\rho(\alpha))$.
(By Theorems 3 and 5 of \cite{S-gen}, $\C[X_G(\Z^N)]$ is generated by functions $\tau_\alpha$ for
 $G=\GL{n}, \SL{n},\Sp{n},\SO{2n+1},$ for all $n.$)
Note that
$$\chi_*(\tau_\alpha)=\begin{cases} \sum_{i=1}^n x_i^\alpha & \text{for\ } G=\GL{n}, \SL{n},\\
\sum_{i=1}^n x_i^\alpha+x_i^{-\alpha} & \text{for\ } G=\Sp{n},\\
\sum_{i=1}^n x_i^\alpha+x_i^{-\alpha}+1, & \text{for\ } G=\SO{2n+1}.\\
\end{cases}$$
Hence $\chi_*(\tau_\alpha)$ is a constant plus a non-zero scalar multiple of $\sum_{w\in W} w\cdot x_i^\alpha.$
Consequently, $\chi_*(\C[X_G(\Z^N)])$ contains all elements of level $1$ in $\C[\T^N]^W.$
Therefore, it is enough to prove that $\C[\T^N]^W$ is generated by such elements. That follows by induction from
Lemma \ref{level_red}.

Let $G=\SO{2n}$ now.
By \cite[Thm 6]{S-gen}, $\C[X_G(\Z^N)]$ is generated by functions $\tau_\alpha$ and by the functions
$Q_{2n}(\alpha_1,...,\alpha_n)$, for $\alpha_1,..,\alpha_n\in \Z^N.$
The homomorphism $\chi_*$ maps $\tau_\alpha$ to $\sum_{i=1}^n (x_i^{\alpha_i}+x_i^{-\alpha_i}).$
Therefore, it is enough to prove that $\C[\T^N]^W$ is generated by elements of level $1$ and by the elements
$\chi_*(Q_{2n}(\alpha_1,...,\alpha_n))$ for $\alpha_1,...,\alpha_n\in \Z^N$. These latter elements are written explicitly in the lemma below. The statement of Proposition \ref{class_onto} follows now by induction
from Lemma \ref{level_red}.
\epr

\blem 
For every $\alpha_1,...,\alpha_n\in \Z^N,$ the function
$$\chi_*(Q_n(\alpha_1,...,\alpha_n)): \T^N/W\to \C$$
is given by
$$i^n\cdot \sum_{\sigma\in S_n}\, sn(\sigma)\prod_{i=1}^n (x_{\sigma(i)}^{\alpha_i}-x_{\sigma(i)}^{-\alpha_i}),$$
where, $sn(\sigma)$ is the sign of $\sigma$ and, as before, $x_k^{\alpha_i}=\prod_{j=1}^N x_{kj}^{\alpha_{ij}}.$
\elem

\bpr
$Q_n(\alpha_1,...,\alpha_n)$ is a composition of two functions. The first one sends $(x_{ij})$ to an element in $\T^N\subset G^N$ whose $k$th component is $(z_{k1},...,z_{kn})=(x_1^{\alpha_k},...,x_n^{\alpha_k})\in (\C^*)^n=\T.$
The second one is a complex valued function on $n$-tuples of matrices in $\SO{2n}$ given by \cite[(2)]{S-gen}:

$Q_n(A,...,Z)=\sum_{\sigma\in S_n} sn(\sigma)
 (A_{\sigma(1),\sigma(2)}-A_{\sigma(2),\sigma(1)})\cdot ...\hspace*{.2in}$\vspace*{-.2in}\\
\be{equation}\label{def-Q}
\hspace*{2.8in}  \cdot (Z_{\sigma(n-1),\sigma(n)}-Z_{\sigma(n),\sigma(n-1)}).
\en{equation}
Since the matrices belonging to the maximal torus $\T$ in $\SO{n}$ are built of diagonal blocks
$$A_j=\frac{1}{2}\left(\be{array}{cc} x_j+x_j^{-1} & i(x_j-x_j^{-1})\\ -i(x_j-x_j^{-1}) & x_j+x_j^{-1}\en{array}\right),$$ for $j=1,...,n,$
$Q_n$ restricted to $n$-tuples of elements of $\T=(\C^*)^n$ sends ($z_{ki}$) to
$$i^n\cdot \sum_{\sigma\in S_n} sn(\sigma) (z_{\sigma(1),1}-z_{\sigma(1),1}^{-1})\cdot ...
\cdot (z_{\sigma(n),n}-z_{\sigma(n),n}^{-1}).$$
Hence, the statement follows.
\epr

\blem\label{level_red}
(1) For $G=\GL{n}, \SL{n}, \Sp{n}, \SO{2n+1}$, every element of $\C[\T^N]^W$ of level $>1$ can be expressed
as a polynomial in elements of $\C[\T^N]^W$ of lower level. The same is true for $G=\SO{2n}$, for elements of $\C[\T^N]^W$ level $1<l<n.$\\
(2) If $G=\SO{2n}$ then every element of $\C[\T^N]^W$ of level $n$ can be expressed
as a linear combination of elements
$$\sum_{\sigma\in S_n}\, sn(\sigma)\prod_{i=1}^n (x_{\sigma(i)}^{\alpha_i}-x_{\sigma(i)}^{-\alpha_i}),$$
for $\alpha_1,...,\alpha_n\in \Z^N,$ and of a polynomial in elements of $\C[\T^N]^W$ of level $<l.$
\elem

\bpr
(1) Consider an element of $\C[\T^N]^W$ of level $l.$ Since it is a linear combination of elements
\beq\label{inv-elt}
\sum_{w\in W} w\cdot m,
\eeq
where $m=\prod_{i=1}^n x_i^{\alpha_i}$ have level $\leq l$, it is enough to prove the statement for such elements.
Since $\sum_{w\in W} w\cdot m$ is invariant under any even permutation of indices $i$, we can assume that
$\alpha_i=0$ for $i> l$. We consider the following three cases separately:

($A_n$) If $G=\GL{n}, \SL{n}$ then
$$\sum_{w\in W} w\cdot m=\sum_{w\in S_{n}} \prod_{i=1}^{l} x_{w(i)}^{\alpha_i}.$$
We have
\beq\label{SL-AB}
\sum_{k=1}^n x_{k}^{\alpha_{l}} \cdot \sum_{w\in S_{n}} \prod_{i=1}^{l-1} x_{w(i)}^{\alpha_i}=
A+B,
\eeq
where $A$ is the sum of the monomials $x_k^{\alpha_{l}}\prod_{i=1}^{l-1} x_{w(i)}^{\alpha_i}$ such that
$$k\in \{w(0),...,w(l-1)\}$$ and $B$ is the sum of the remaining ones.
Note that $$B= (n-l) \sum_{w\in S_{n}} \prod_{i=1}^{l} x_{w(i)}^{\alpha_i}$$
is a non-zero multiple of (\ref{inv-elt}).
Since $A$ is an element of $\C[\T^N]^W$ of level $< l$ and the left hand side of (\ref{SL-AB}) is a product
of elements of $\C[\T^N]^W$ of level $< l$, the statement follows.

($B_n$+$C_n$) If $G=\SO{2n+1}$ or $\Sp{n}$ then
$$\sum_{w\in W} w\cdot m=\sum_{w\in S_n}\sum_{\ve_1,...,\ve_n\in \{+1,-1\}} \prod_{i=1}^{l} x_{w(i)}^{\ve_i\cdot \alpha_i}$$
and
\beq\label{SpSO-AB}
\sum_{k=0}^n \left( x_{k}^{\alpha_l}+x_{k}^{-\alpha_l} \right) \cdot \sum_{w\in S_n} \sum_{\ve_1,...,\ve_n\in \{+1,-1\}} \prod_{i=1}^{l-1} x_{w(i)}^{\ve_i\cdot \alpha_i}=
A+B,
\eeq
where $A$ is the sum of the monomials
$$x_k^{\pm \alpha_l}\prod_{i=1}^{l-1} x_{w(i)}^{\pm \alpha_i}$$
such that
$$k\in \{w(0),...,w(l-1)\}$$ and $B$ is the sum of the remaining ones.
Note that $$B= (n-l) \sum_{w\in S_n}\sum_{\ve_1,...,\ve_n\in \{+1,-1\}} \prod_{i=1}^{l} x_{w(i)}^{\ve_i\cdot \alpha_i}$$ is a non-zero multiple of (\ref{inv-elt}).
Since $A$ is an element of $\C[\T^N]^W$ of level $< l$ and the left hand side of (\ref{SpSO-AB}) is a product of elements of $\C[\T^N]^W$ of level $< l$, the statement follows.

($D_n$) Let $G=\SO{2n}$. Since $m$ has level $<n$ and the negation of the sign of a missing variable does not affect $m$,
$$\sum_{w\in W} w\cdot m=\frac{1}{2}\sum_{w\in SS_n} w\cdot m.$$
Therefore the statement follows from the argument for the $(B_n)$ case.

(2) As before, since every element of $\C[\T^N]^W$ of level $n$ is a linear combination of elements
\beq\label{inv-elt2}
\sum_{w\in W} w\cdot m,
\eeq
where $m=\prod_{i=1}^n x_i^{\alpha_i}$ have level $\leq n$, it is enough to prove the statement for elements of level $n$.

Let $\ve(w)=\pm 1$ for $w\in SS_n$ depending on whether the number of sign changes in $w$ is even or odd.
Then $$\sum_{w\in W} w\cdot m=\frac{1}{2}\sum_{w\in SS_n} w\cdot m+\frac{1}{2}\sum_{w\in SS_n} w\cdot \ve(w) m.$$
Since the first summand on the right is $SS_n$ invariant, it is a polynomial in elements of $\C[\T^N]^W$ of lower
level by the argument for ($B_n$) above. The second summand is equal to
$$\frac{1}{2} \sum_{\sigma\in S_n} sn(\sigma) \prod_{i=1}^n (x_{\sigma(i)}^{\alpha_i}-x_{\sigma(i)}^{-\alpha_i}).$$
\epr


%
\section{Proofs of Remark \ref{ir=con}(3) and of Proposition \ref{irred}}
\label{s_proofs2}
%

\noindent {\bf Proof of Remark \ref{ir=con}(3)} following \cite{Th}: For any connected reductive group $G,$ there is an epimorphism $$C^c(G)\times [G,G]\to G,$$
with a finite kernel, where $C^c(G)$ is the connected component of the identity in the center of $G,$ cf. \cite[Prop IV.14.2]{B2},
Since $[G,G]$ is semi-simple, it has a finite cover $G'$ which is simply-connected.
Hence we have a finite extension of $G$:
\begin{equation}\label{nu}
\{1\}\to K \to C^c(G)\times G' \stackrel{\nu}{\longrightarrow} G\to \{1\}.
\end{equation}
By \cite[Thm. C]{Ric}, $Hom(\Z^2,G')$ is irreducible. Since $C^c(G)$ is irreducible,
$$Hom(\Z^2,C^c(G)\times G')=(C^c(G))^2\times Hom(\Z^2,G')$$
is irreducible as well. Let $Hom^c(\Z^2,G)\subset Hom(\Z^2,G)$ be the connected component of the trivial representation and let $$\nu_*: Hom(\Z^2,C^c(G)\times G')\to Hom^c(\Z^2,G)$$ be the morphism induced by $\nu.$
By the lemma below, $Hom^c(\Z^2,G)$ is irreducible. Hence, $X_G^c(\Z^2)$ is irreducible as well and, therefore, it coincides with
$X_G^0(\Z^2).$
\qed\\

\begin{lemma}
$\nu_*: Hom(\Z^2,C^c(G)\times G')\to Hom^c(\Z^2,G)$ is onto.
\end{lemma}

\begin{proof}
We need to prove that every representation $f$ in $Hom^c(\Z^2,G)$ lifts to a representation $\tilde f: \Z^2\to C^c(G)\times G'$ (i.e. $f=\nu \tilde f$).
Since the extension (\ref{nu}) is finite and central, it defines an element $\alpha\in H^2(G,K)$ such that
$f$ lifts to $\tilde f$ if and only if $f^*(\alpha)=0$ in $H^2(\Z^2,K),$ cf. \cite[Sec. 2]{GM}.
Since $H^2(\Z^2,K)$ is discrete, the property of $f$ being ``liftable" is locally constant on $Hom(\Z^2,G)$
(in complex topology) and, hence, constant on $Hom^c(\Z^2,G).$ Since the trivial representation is liftable, the statement follows.
\end{proof}

The following will be needed for the proof of Proposition \ref{irred}:

\bpro \lb{Borel}
If the image of $\rho: \Z^N\to G$ belongs to a Borel subgroup of $G$ then $\rho$ is equivalent in $X_G(\Z^N)$ to a representation with an image in a maximal torus of $G.$
\epro

\bpr
Any Borel subgroup $B\subset G$ is of the form $\T\cdot U,$ where $\T$ is a maximal torus and $U$ is a unipotent subgroup of $G.$ Let $e_1,...,e_N$ be generators of $\Z^N$ and let $t_i\cdot u_i$ be a decomposition of $\rho(e_i).$
By \cite[Prop. 8.2.1]{Sp}, $U$ is generated by rank $1$ subgroups $U_\alpha$, associated with roots $\alpha$ which are positive with respect to some ordering. Furthermore, it follows from \cite[Prop. 8.1.1]{Sp}, there exists a sequence $s_1, s_2,...\in \T$ which conjugates $u_1,...,u_N$ to elements arbitrarily close to $\T.$ Consequently, $s_n\rho(e_i)s_n^{-1}\to t_i$ as $n\to \infty$ for every $i=1,...,n.$ Equivalence classes of representations in $Hom(\Z^N,G)$ are closed in Zariski topology and, hence, in complex topology as well. Therefore $\rho$ is equivalent to $\rho'$ sending $e_i$ to $t_i$ for every $i.$
\epr

\noindent {\bf Proof of Proposition \ref{irred}:}
Let $G=\GL{n}$ or $\SL{n}.$ Since the matrices $\rho(e_1),\rho(e_2),...,\rho(e_N)\in G$ commute, they can be simultaneously conjugated to upper triangular ones and, hence, they lie in a Borel subgroup of $G.$ Now the statement follows from Proposition \ref{Borel}.

The same holds for $G=\Sp{n}:$ Recall that a subspace $V$ of a symplectic space $\C^{2n}$ is isotropic if the symplectic form restricted to $V$ vanishes. A stabilizer of any complete flag
$\{0\}=V_0\subset V_1\subset ...\subset V_n$ of isotropic subspaces of $\C^{2n}$ is a Borel subgroup of $\Sp{n},$ cf. \cite[Ch. 10]{GW}. Therefore, to complete the proof, it is enough to show the existence of a complete isotropic flag preserved by $\rho(\Z^N).$ We construct it inductively. Let $V_0=\{0\}.$ Suppose that $V_k$ is defined already. Then $\rho(\Z^N)$ preserves $V_k^\perp$.
Since any number of commuting operators on a complex vector space preserves a $1$-dimensional subspace, there is such subspace  $W\subset V_k^\perp/V_k,$ as long as $V_k$ is not a maximal isotropic subspace. Let $V_{k+1}=\pi^{-1}(W)$ then, where $\pi$ is the projection $V_k^\perp\to V_k^\perp/V_k.$
\qed \ \\

%
\section{Proof of Theorem \ref{slice}, Corollary \ref{T=H}, and Proposition \ref{om=om'}}
\label{s_proofs3}
%

\noindent {\bf Proof of Theorem \ref{slice}:}
(1) The argument of the proof of Theorem \ref{main}(1) shows that (\ref{e-T}) is an equality and, therefore, $\rho$ is a simple point of $Hom(\Z^N,G)$. By  \cite[II \S2 Thm 6]{Sh}, $\rho$ belongs to a unique component.

(2) Consider the map $\lambda: X_G(\Z^N)\to X_G(\Z)\times ...\times X_G(\Z),$ sending $[\rho]$ to the $N$-tuple $([\rho(e_1)],...,[\rho(e_N)])$. Since the composition
$$\T^N\to \T^N/W\xrightarrow{\chi} X_G(\Z^N)\xrightarrow{\lambda} X_G(\Z)\times ...\times X_G(\Z)$$
is the Cartesian product of the maps $\T\to X_G(\Z)=\T/W,$ its differential is onto. Hence $d\chi$ has rank $N\cdot rank\, G$, which implies that $d\chi$ is $1$-$1$. By (1) and by Theorem \ref{main}(1), $dim\, T_\rho\, X_G(\Z^N)=N\cdot rank\, G.$ Therefore, $d\chi$ is an isomorphism.

(3) follows from \cite[Prop. 4.18]{Dr} (cf. the argument of the proof of Luna \'Etale Slice Theorem in \cite{Dr}).
\qed\ \\

\noindent {\bf Proof of Corollary \ref{T=H}:} By Lemma \ref{H^1-Z^N}(2), there is a natural identification of $\t^N=H^1(\Z^N, \t)$ with $H^1(\Z^N, Ad\, \rho).$
Since the resulting isomorphism
$$H^1(\Z^N, Ad\, \rho)\to \t^N\to T_\rho\, \T^N/W\to T_\rho\, X_G(\Z^N)$$ coincides with (\ref{e_H^1-T}), the statement follows.
\qed\ \\

\noindent {\bf Proof of Proposition \ref{om=om'}:}
(1) By Lemma \ref{H^1-Z^N}, $H^1(\Z^2,Ad\, \rho)=H^1(\Z^2,\t)$ for $[\rho]\in X_G'(\Z^2)$ (i.e. in the domain of $\omega$). Since the cup product
$$H^1(\Z^2,\t)\times H^1(\Z^2,\t)\stackrel{\cup}{\longrightarrow} H^2(\Z^2,\t\otimes \t)=\t\otimes \t$$
sends $(v_1,w_1),(v_2,w_2)$ to $v_1\otimes w_2 -v_2\otimes w_1,$
 the statement follows.

(2a) Any triple of vectors in any tangent space to $\T\times \T$
extends to invariant vector fields $X_1,X_2,X_3$ on $\T\times \T.$
Since $d\omega'(X_1,X_2,X_3)$ is a linear combination of terms $X_i(\omega'(X_j,X_k))$ and
$\omega'([X_i,X_j],X_k),$ it vanishes for such fields. Therefore, $\omega'$ is closed. Being non-degenerate, it is also symplectic.

(2b) Since $\chi^*(d\omega)=d(\chi^*\omega)=d\omega'=0,$ and $\chi^*$ (being a normalization map) is an isomorphism of tangent spaces on a Zariski dense subset of $X_G'(\Z^N)$, $d\omega=0$ on $X_G'(\Z^N)$.
By its construction, $\omega$ is non-degenerate and, hence, symplectic.
\qed\ \\

%
\section{Proof of Proposition \ref{G-brack}.}
\lb{s_proof_symp}
%

In the statement below, the notion of Goldman bracket refers to the Poisson bracket dual to the holomorphic Goldman symplectic form defined by (\ref{e-omega}), where $\frak B=c\cdot \frak T$, $c\in \C^*$ and $\frak T$ is the trace form, as in Sec. \ref{s_symp}.

\bpro \lb{Goldman-formulas}
The following formulas hold for Goldman brackets for all closed orientable surfaces of genus $\geq 1$:\\
(1) For $G=\SL{n},$
\begin{equation}\label{sln_brack}
\{\tau_\alpha,\tau_\beta\}= \frac{1}{c} \sum_{p\in \alpha \cap \beta} \ve(p,\alpha,\beta) \left(\tau_{\alpha_p\beta_p}- \frac{\tau_{\alpha}\tau_{\beta}}{n}\right),
\end{equation}
where $\alpha,\beta$ are any smooth closed oriented loops in $F$ in general position.
(We identify closed oriented loops in $F$ with conjugacy classes in $\pi_1(F).$)
$\alpha \cap \beta$ is the set of the intersection points and
$\alpha_p\beta_p$ is the product of $\alpha$ and $\beta$ in $\pi_1(F,p),$
and $\ve(p,\alpha,\beta)$ is the sign of the intersection:\vspace*{.1in}

\centerline{\parbox{2in}{\psfig{figure=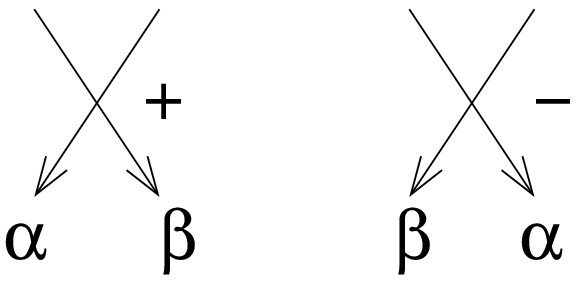,height=.5in}}}
\vspace*{.15in}

(2) For $G=\SO{n}, \Sp{n},$
\begin{equation}\label{spso_brack}
\{\tau_{\alpha},\tau_{\beta}\}= \frac{1}{2c} \sum_{p\in \alpha \cap \beta} \ve(p,\alpha,\beta)
\left(\tau_{\alpha_p\beta_p}-\tau_{\alpha_p\beta_p^{-1}}\right).
\end{equation}
\epro

Since the signed number of the intersection points between any two curves $(p,q)$, $(r,s)$ in a torus is $\left|\begin{array}{cc} p & q \\ r & s\\ \end{array}\right|,$ the above statement immediately implies Proposition \ref{G-brack}.

The proof of Proposition \ref{Goldman-formulas} uses the notion of variation function introduced in \cite{Go2}.
Let $F:G\to \g$ be the variation function with respect to $\frak B=c\cdot \frak T,$ $c\in \C^*.$

\blem \lb{l_var}
Consider the standard embeddings $\SL{n},\SO{n}\subset \GL{n},$ $\Sp{n}\subset \GL{2n},$ and the induced embeddings of Lie algebras. Then\\
(1) $F:\SL{n}\to \sl{n}\subset \M{n}$ is given by $F(A)=\frac{1}{c}(A-\frac{Tr(A)}{n}I)$\\
(2) 
$F:\SO{n}\to \so{n} \subset \M{n}$ is given by $F(A)=\frac{1}{2c}(A-A^{-1})$, and\\
(3) $F:\Sp{n}\to \sp{n} \subset \M{2n}$ is given by
$F(A)=\frac{1}{2c}(A-A^{-1})$.
\elem

\bpr By its definition, the variation function with respect to $c\cdot \frak T,$ is $c^{-1}$ times the variation function with respect to $\frak T.$ Therefore, it is enough to prove the statement for $c=1.$

It is easy to see that the following ``complex" version of \cite[Sec 1.4]{Go2} holds: In the above setting, the variation function is given by the composition $$G\to \GL{n}\to \M{n} \xrightarrow{pr} \g,$$
where $pr$ is the orthogonal projection with respect to $\frak T.$ Indeed, Goldman's proof of the ``real" version carries over the complex case. Now the statement
follows from computations like those of Corollaries 1.8 and 1.9 of \cite{Go2}.
\epr

\noindent{\bf Proof of Proposition \ref{Goldman-formulas}:}
By Goldman's Product Formula, \cite{Go2},
$$\{\tau_{\alpha},\tau_{\beta}\}([\rho])= \sum_{p\in \alpha \cap \beta}
\ve(p; \alpha,\beta)  \frak B(F_{\alpha_p}(\rho_p),F_{\beta_p}(\rho_p)),$$
where $\rho_p$ is the $G$-representation of $\pi_1(F,p)$ which belongs to the conjugacy class $[\rho].$

By Lemma \ref{l_var}(1), for $G=\SL{n},$
$$ \frak B(F_{\alpha_p}(\rho_p),F_{\beta_p}(\rho_p))=$$
$$c\cdot Tr\left(\frac{1}{c^2}\left(\rho_p(\alpha_p)-\frac{Tr(\rho_p(\alpha_p))}{n}I\right)
\left(\rho_p(\beta_p)-\frac{Tr(\rho_p(\beta_p))}{n}I\right)\right)=$$
$$\frac{1}{c} \left(Tr(\rho_p(\alpha_p\beta_p))-\frac{Tr(\rho_p(\alpha_p))Tr(\rho_p(\beta_p))}{n}\right)$$
and Proposition \ref{Goldman-formulas}(1) follows.

An analogous computation using Lemma \ref{l_var}(2) and (3) implies part (2).
\qed

%

\end{document}